\documentclass[11pt
]{article}

\usepackage[all]{xy}

\usepackage[latin1]{inputenc}
\usepackage{amsfonts}
\usepackage{amsmath}
\usepackage{amssymb}
\usepackage{amscd}
\usepackage{amsthm}
\usepackage{stmaryrd}
\usepackage{indentfirst}
\usepackage[hmargin=2.7cm
,vmargin=3.5cm
]{geometry}

\usepackage{faktor}

\usepackage{epsfig}

\newtheorem{thm}{Theorem}[section]
\newtheorem{lemma}[thm]{Lemma}

\newtheorem{prop}[thm]{Proposition}
\newtheorem{coroll}[thm]{Corollary}

\theoremstyle{definition}

\newtheorem{rem}[thm]{Remark}
\newtheorem{exam}[thm]{Example}

\newtheorem{notation}[thm]{Notation}
\newtheorem*{acknow}{Acknowledgements}

\newcommand{\R}{{\mathbb{R}}}
\newcommand{\F}{{\mathbb{F}}}

\newcommand{\Z}{{\mathbb{Z}}}

\newcommand{\C}{{\mathbb{C}}}

\newcommand{\cA}{{\mathcal{A}}}

\newcommand{\cL}{{\mathcal{L}}}

\newcommand{\cP}{{\mathcal{P}}}
\newcommand{\cQ}{{\mathcal{Q}}}
\newcommand{\cR}{{\mathcal{R}}}

\newcommand{\fc}{{:\ }}

\newcommand{\ol}{\overline}
\newcommand{\wt}{\widetilde}

\newcommand{\tb}{\textbf}

\DeclareMathOperator{\IM}{Im}

\DeclareMathOperator{\Crit}{Crit}

\DeclareMathOperator{\id}{id}

\DeclareMathOperator{\chr}{char}

\DeclareMathOperator{\pt}{pt}

\DeclareMathOperator{\Ham}{Ham}

\DeclareMathOperator{\Aut}{Aut}

\DeclareMathOperator{\Spec}{Spec}

\DeclareMathOperator{\Pin}{Pin}

\begin{document}

\title{On the Lagrangian Hofer geometry of Clifford tori}

\author{Frol Zapolsky}


\setcounter{tocdepth}{3}

\renewcommand{\labelenumi}{(\roman{enumi})}

\maketitle


\begin{abstract}
  We show that the space of Lagrangians which are Hamiltonian isotopic to the Clifford torus in a complex projective space or in the four-dimensional quadric, taken with Chekanov's Lagrangian Hofer metric, contains a quasi-isometric copy of the real line, and in particular has infinite diameter.
\end{abstract}

\section{Main result}

The celebrated Hofer distance on the Hamiltonian group $\Ham(M,\omega)$ of a symplectic manifold $(M,\omega)$ is defined by
$$d_{\mathrm H}(\phi,\psi)=\|\phi^{-1}\psi\| \qquad\text{for }\phi,\psi \in \Ham(M,\omega)\,,$$
where
$$\|\phi\| = \inf_{\phi=\phi_H}\int_0^1\Big(\max_MH_t - \min_MH_t\Big)\,dt$$
is the Hofer norm of $\phi \in \Ham(M,\omega)$, $\phi_H$ being the time-$1$ map of the Hamiltonian flow $\phi_H^t$ of a given Hamiltonian $H$. It is a group pseudo-norm in the sense of \cite{Burago_Ivanov_Polterovich_Conj_Invt_Norms_Groups_Geom_Origin}.

The Hofer distance has been a central theme in symplectic geometry ever since its inception in \cite{Hofer_Top_Props_Symplectic_Maps}. It is easily shown to be a biinvariant pseudometric. However, its nondegeneracy is quite a nontrivial property, and was established in increasing levels of generality in \cite{Hofer_Top_Props_Symplectic_Maps}, \cite{Polterovich_Symplectic_Displacement_Energy_Lag_Submfds}, \cite{Lalonde_McDuff_Geom_Sympl_Energy}. We will not attempt to cover the extensive literature on the subject, instead referring the reader to \cite{Polterovich_Geom_Group_Sympl_Diffeos} for basics on Hofer geometry.

The goal of this note is a result about the Lagrangian analog of the Hofer distance, introduced by Chekanov in \cite{Chekanov_Invt_Finsler_Metrics_Space_Lagrangians}. Let $(M,\omega)$ be a connected symplectic manifold, and let $\Ham(M,\omega)$ be the group generated by compactly supported Hamiltonians. Let $L\subset M$ be a closed connected Lagrangian submanifold. Throughout we denote by
$$\cL(L) = \{\phi(L)\,|\,\phi\in\Ham(M,\omega)\}$$
the collection of Lagrangians in $M$ which are Hamiltonian isotopic to $L$. For $L',L'' \in \cL(L)$ put
$$d_{\mathrm{LH}}(L',L'') = \inf\{\|\phi\|\,|\,\phi(L')=L''\}\,.$$
This is clearly a pseudometric on $\cL(L)$ which is invariant under the obvious action of $\Ham(M,\omega)$. Chekanov showed in \cite{Chekanov_Invt_Finsler_Metrics_Space_Lagrangians} that it is nondegenerate provided $M$ is geometrically bounded, and also furnished a counterexample which shows that this condition cannot be dropped. Note that, in contrast, the Hofer distance on the Hamiltonian group is \emph{always} nondegenerate.

A basic question regarding any metric space is whether it has infinite diameter. In Hofer geometry, thanks to the nature of symplectic topology, such a result has always been a consequence of a stronger statement, namely that there is a quasi-isometric copy of the real line in the relevant space. Our purpose here is to prove the following.
\begin{thm}\label{thm:main}
  Let $(M,\omega)$ be either $S^2\times S^2$ with the standard monotone symplectic form or $\C P^n$ with the Fubini--Study form, and let $T\subset M$ be the Clifford torus. Then $(\cL(T),d_{\mathrm{LH}})$ contains a quasi-isometric copy of $\R$. More precisely, there is a family $(L_t)_{t\in \R}\subset \cL(T)$ and a constant $D>0$ such that for all $t,s \in \R$ we have
  $$\tfrac 12 |s-t| - D \leq d_{\mathrm{LH}}(L_s,L_t) \leq |s-t|\,.$$
  In particular $\cL(T)$ has infinite diameter.
\end{thm}
\noindent To the best of our knowledge, this is the first instance of such a result for a positively monotone Lagrangian in a closed symplectic manifold, which is not obtained through stabilization. It is plausible that the same result holds for other Lagrangians in similar settings, however finding examples is quite challenging. Previous results regarding the infinity of the diameter of the Lagrangian Hofer distance include \cite{Oh_Sympl_Top_Geom_Action_Fcnl},  \cite{Khanevsky_Hofers_Metric_Space_Diameters}, \cite{Zapolsky_Hofer_Geom_Weakly_Exact_Lags}, \cite{Seyfaddini_Unboundedness_Lag_Hofer_Dist_Euclidean_Ball}, \cite{Usher_Hofer_Geom_Cot_Fibers}, \cite{Masatani_Family_Lag_Submfds_Bidisks_Lag_Hofer_Metric}, \cite{Gong_Unbounded_Lagrangian_Spectral_Norm_Wrapped_Floer_Coh}, \cite{Trifa_Hofer_Dist_Lag_Links_Disc}, \cite{Dawid_Hofer_Geom_A3_Configs}.

\begin{acknow}This paper arose as a result of collaboration with Michael Khanevsky. I am grateful to him for many stimulating discussions and for explaining the fundamental technique of \cite{Khanevsky_Hofers_Metric_Space_Diameters}. I would like to warmly thank Egor Shelukhin for his interest in this project, for valuable discussions, and for directing my attention to a higher-dimensional analog of Chekanov's torus in $\C P^2$.
\end{acknow}

\section{Proof}

\subsection{Quasi-morphisms}

As usual, the upper bound in Theorem \ref{thm:main} is a straightforward consequence of definitions. The fundamental idea in the proof of the lower bound is to use quasi-morphisms on the Hamiltonian group which are Lipschitz with respect to the Hofer distance \cite{Khanevsky_Hofers_Metric_Space_Diameters}. Let us present the relevant discussion in an abstract case for the convenience of the reader.

Recall that a quasi-morphism on a group $G$ is any function $\mu \fc G\to\R$ such that
$$D(\mu):=\sup_{a,b\in G}|\mu(ab) - \mu(a) - \mu(b)| < \infty\,.$$
The quantity $D(\mu)$ is called the defect of $\mu$. We say that $\mu$ is homogeneous if in addition
$$\mu(a^n) = n\mu(a)\qquad\text{for all }a\in G\,,n\in\Z\,.$$
Given a quasi-morphism $\mu$, there exists a unique homogeneous quasi-morphism $\mu_\infty$ which lies at a bounded distance from $\mu$, that is $\sup_{a\in G}|\mu(a)-\mu_{\infty}(a)| < \infty$. This $\mu_\infty$ is called the homogenization of $\mu$ and it is given by the formula
$$\mu_\infty(a) = \lim_{n\to\infty}\frac{\mu(a^n)}{n}\,.$$
See \cite{Calegari_scl} for basics on quasi-morphisms.

Let us now see how quasi-morphisms yield lower bounds on certain pseudometrics. Let $G$ be a group equipped with a metric $d$, and let $H<G$ be a subgroup. Consider the associated homogeneous space $G/H$ with the induced pseudometric
$$\ol d(gH,g'H) = \inf\{d(gh,g'h')\,|\,h,h'\in H\}\,.$$
Now assume we are given a quasi-morphism $\nu \fc G\to \R$ which is $\lambda$-Lipschitz with respect to $d$ for some $\lambda > 0$, that is for all $g,g' \in G$:
$$|\nu(g)-\nu(g')|\leq \lambda d(g,g')\,.$$
Assume, moreover, that $\nu|_H\equiv 0$. For $g,g'\in G$ and $h,h' \in h$ we then have
\begin{align*}
  \lambda d(gh,g'h') &\geq |\nu(gh) - \nu(g'h')|\\
                     &\stackrel{*}{=} |(\nu(g) - \nu(g')) + (\nu(gh)-\nu(g)-\nu(h)) - (\nu(g'h') - \nu(g') - \nu(h')|\\
                     &\geq |\nu(g) - \nu(g')| - |\nu(gh)-\nu(g)-\nu(h)| - |\nu(g'h') - \nu(g') - \nu(h')|\\
                     &\geq |\nu(g) - \nu(g')| -2D(\nu)\,,
\end{align*}
where $\stackrel{*}{=}$ holds because $\nu(h) = \nu(h') = 0$. Taking the infimum over $h,h' \in H$ we deduce:
$$\frac 1 \lambda |\nu(g)-\nu(g')| - \frac{2D(\nu)}{\lambda}\leq \ol d(gH,g'H)\,.$$

Theorem \ref{thm:main} is then an easy consequence of the following result:
\begin{prop}\label{prop:qm}
  Let $(M,\omega)$ and $T$ be as in Theorem \ref{thm:main}. Then there exist a quasi-morphism $\nu \fc \Ham(M,\omega)\to \R$ which is $2$-Lipschitz with respect to the Hofer distance, and a Hamiltonian $H \in C^\infty(M,[0,1])$ such that
  \begin{itemize}
    \item $\nu(\phi^t_H) = t$ for all $t\in\R$;
    \item $\nu(\psi)=0$ for all $\psi$ satisfying $\psi(T) = T$.
  \end{itemize}
\end{prop}
\noindent Let us prove Theorem \ref{thm:main}, assuming the proposition.
\begin{proof}[Proof of Theorem \ref{thm:main}]
  Set $G = \Ham(M,\omega)$ and $H = \{\psi \in G\,|\,\psi(T) = T\}$. It then follows that
  $$G/H \to \cL(T)\,,\quad gH\mapsto g(T)$$
  is a well-defined bijection. Moreover, letting $d$ be the Hofer distance on $G$, we see that under this bijection the induced pseudometric $\ol d$ on $G/H$ coincides with the Lagrangian Hofer distance on $\cL(T)$. It follows that the discussion at the beginning of this section applies to the present situation. Namely, we can deduce that for all $g,g' \in G$ we have
  $$\tfrac12|\nu(g) - \nu(g')| - D(\nu) \leq d_{\mathrm{LH}}(g(T),g'(T))\,.$$
  Now set $L_t=\phi^t_H(T)$. Then
  $$\tfrac 12|s-t| - D(\nu) = \tfrac12|\nu(\phi^s_H) - \nu(\phi^t_H)| - D(\nu) \leq d_{\mathrm{LH}}(L_s,L_t)\,,$$
  which is the left-hand side of the required inequality. The right-hand side follows immediately from the definition of the Hofer distance:
  $$d_{\mathrm{LH}}(L_s,L_t) \leq d_{\mathrm{H}}(\phi^s_H,\phi^t_H) = \|\phi^{s-t}_H\| \leq \int_0^1\big(\max (s-t)H - \min (s-t) H\big)\,dt = |s-t|\,,$$
  by our assumptions on $H$. The proof is complete.
\end{proof}

\subsection{Hamiltonian spectral invariants}

The rest of the paper is dedicated to proving Proposition \ref{prop:qm}. This requires Hamiltonian and Lagrangian spectral invariants, as well as a few preparatory results.

In this section $(M,\omega)$ is as in Theorem \ref{thm:main}. In order to construct Hofer-Lipschitz quasi-morphisms on $\Ham(M,\omega)$, we will use spectral invariants, following the basic idea of Entov--Polterovich \cite{Entov_Polterovich_Calabi_QM_QH}.

Let us fix a unital commutative ground ring $R$. Let $QH_*(M;R)$ be the quantum homology algebra of $M$ with coefficients in $R$. This has different flavors in different contexts, and for us
$$QH_*(M;R) = H_*(M;R)\otimes_R R[q,q^{-1}]\,,$$
where $q$ is a quantum variable, which has degree $|q| = -2N_M$, where $N_M$ is the minimal Chern number of $M$. The quantum product $*$ is the usual intersection product on singular homology, deformed by contributions from three-point Gromov--Witten invariants \cite{mcduff2012j}. It has degree $-\dim M$, and the fundamental class $[M]$ serves as the unit.

In the cases relevant to this paper, the quantum product is completely determined by the following relations, in addition to the usual intersection product \cite{mcduff2012j}, \cite{Entov_Polterovich_Calabi_QM_QH}:
\begin{itemize}
  \item If $M=\C P^n$, then
    $$[\C P^{n-1}]^{*(n+1)} = q[M]\,.$$
  \item If $M=S^2\times S^2 = \C P^1 \times \C P^1$, then
    $$QH_*(M;R) \cong QH_*(\C P^1;R)\otimes_{R[q,q^{-1}]} QH_*(\C P^1;R)$$
    as algebras. The only relevant relation for us here is given by
    $$\pt_M^{*2} = (\pt_{\C P^1}\otimes\pt_{\C P^1})^{*2} = (\pt_{\C P^1})^{*2}\otimes(\pt_{\C P^1})^{*2} = q[\C P^1]\otimes q[\C P^1] = q^2[M]\,.$$
    It follows that in case $2 \in R$ is invertible, the elements
    $$\frac{[M]\pm q^{-1}\pt}2$$
    are idempotent, that is they coincide with their quantum squares.
\end{itemize}

\begin{notation}
  We denote the above idempotents by
  $$e^R_\pm:=\frac{[M]\pm q^{-1}\pt}2 \in QH_4(S^2\times S^2;R)\,.$$
\end{notation}

Let us denote by $\wt\Ham(M,\omega)$ the universal cover of $\Ham(M,\omega)$, and let $\wt d_{\mathrm H}$ be the corresponding Hofer pseudometric:
$$\wt d_{\mathrm H}(\wt\phi\wt\psi)=\|\wt\phi^{-1}\wt\psi\|\,,\quad\text{where}\quad \|\wt\phi\|=\inf_{\wt\phi=\wt\phi_H}\int_0^1\Big(\max_MH_t - \min_MH_t\Big)\,dt\,,$$
with $\wt\phi_H \in \wt\Ham(M,\omega)$ being the class of the isotopy $(\phi_H^t)_{t\in[0,1]}$.

Recall that the Hamiltonian (that is, closed-string) spectral invariant is a function
$$c^R \fc QH_*(M;R) \times \wt\Ham(M,\omega) \to \R\cup\{-\infty\}\,,$$
satisfying, among others, the following properties:
\begin{itemize}
  \item \tb{(finiteness):} $c^R(A,\wt\phi) \in \R$ if $A \neq 0$;
  \item \tb{(triangle inequality):} $c^R(A*B,\wt\phi\wt\psi) \leq c^R(A,\wt\phi) + c^R(B,\wt\psi)$;
  \item \tb{(Hofer continuity):} $|c^R(\wt\phi)-c^R(\wt\psi)| \leq \wt d_{\mathrm H}(\wt\phi,\wt\psi)$.
\end{itemize}
\noindent This is by no means an exhaustive list, however it will suffice for our proof. Hamiltonian spectral invariants were constructed in the required generality in \cite{Oh_Spectral_Invts}.

\begin{notation}
  We denote $c^R_+:=c^R([M],\cdot)$.
\end{notation}

We have the following fundamental result regarding spectral invariants, proved in \cite{Entov_Polterovich_Calabi_QM_QH}:
\begin{prop}\label{prop:EP_sp_invts_Lipschitz_qms}
  \begin{enumerate}
    \item If $M = \C P^n$ and $\F$ is a field, then $c^\F_+$ is a quasi-morphism;
    \item If $M=S^2 \times S^2$ and $\F$ is a field with $\chr \F \neq 2$, then $c^\F(e^\F_\pm,\cdot)$ are quasi-morphisms. \qed
  \end{enumerate}
\end{prop}

\begin{notation}
  We let
  $$\mu^\F = \big(c^\F_+\big)_\infty \fc \wt\Ham(\C P^n) \to \R \quad\text{and} \quad \mu_\pm^\F = \big(c^\F(e^\F_\pm;\cdot)\big)_\infty \fc \wt\Ham(S^2\times S^2)\to \R$$
  denote the corresponding homogenized quasi-morphisms.
\end{notation}
\begin{rem}\label{rem:homog_Lipschitz_qm_again_Lipschitz}
  Note that if $G$ is a group, $d$ is a biinvariant pseudometric on $G$, and $\nu \fc G \to \R$ is a quasi-morphism which is $\lambda$-Lipschitz with respect to $d$, then so is its homogenization $\nu_\infty$. Indeed, thanks to the biinvariance of $d$, we have for $a,b,g,h \in G$:
  $$d(ab,gh) \leq d(ab,ah) + d(ah,gh) = d(b,h) + d(a,g)\,,$$
  and consequently by induction
  $$d(a^n,b^n)\leq nd(a,b)\,.$$
  It follows that
  $$\frac 1\lambda|\nu_\infty(a) - \nu_\infty(b)| = \lim_{n\to \infty} \frac 1 {\lambda n}|\nu(a^n) - \nu(b^n)| \leq \lim_{n\to \infty} \frac {d(a^n,b^n)}n \leq \lim_{n\to \infty} d(a,b) = d(a,b)$$
\end{rem}
\noindent In \cite{Entov_Polterovich_Calabi_QM_QH} Entov--Polterovich show that the homogenized quasi-morphisms $\mu^\F,\mu^\F_\pm$ descend from $\wt\Ham(M,\omega)$ to $\Ham(M,\omega)$. Combining this with Remark \ref{rem:homog_Lipschitz_qm_again_Lipschitz}, we obtain:
\begin{prop}\label{prop:EP_mu_Lipschitz_qms_descend_Ham}
  The quasi-morphisms $\mu^\F,\mu^\F_\pm$ descend to $\Ham(M,\omega)$, where they are $1$-Lip\-schitz with respect to the Hofer distance. \qed
\end{prop}

\subsection{Lagrangian spectral invariants}

The last piece of the puzzle is Lagrangian spectral invariants as defined in \cite{leclercq2018spectral}. These play a central role in the proof of the main theorem, where the idea is that the quasi-morphisms $\mu^\F,\mu^\F_\pm$ from the previous section can be expressed in terms of Lagrangian spectral invariants for suitable Lagrangians.

Let us describe the relevant setting and summarize the necessary results. Let $L\subset M$ be a closed connected Lagrangian, which is assumed to be positively monotone, meaning there is $\tau > 0$ such that the symplectic area $\omega \fc \pi_2(M,L)\to \R$ and the Maslov index $\mathsf m \fc \pi_2(M,L) \to \Z$ satisfy
$$\omega = \tau\cdot \mathsf m\,.$$
We furthermore assume that the minimal Maslov number of $L$, that is
$$N_L:=\inf \{\mathsf m(A) \,|\, A\in \pi_2(M,L)\,,\mathsf m(A) > 0\}\,,$$
is at least $2$. The last assumption is that $L$ is either relatively $\Pin^+$ or relatively $\Pin^-$, see \cite{Solomon_Intersection_Thry_Moduli_Space_Holo_Curves_Lag_Bd_Cond}, \cite{Wehrheim_Woodward_Orientations_Quilts}; if either one of these cases holds, we say that $L$ is relatively $\Pin$.

\begin{center}
  \emph{In what follows, all Lagrangians are assumed closed, connected, monotone, and relatively $\Pin$.}
\end{center}

Fix a ground ring $R$. Under our assumptions, one can define the Lagrangian quantum homology $QH_*(L;R)$ with coefficients in $R$ \cite{biran2007quantum}, \cite{Biran_Cornea_Lag_Top_Enumerative_Geom}, \cite{Zapolsky_Orientations}. It is an associative unital $R$-algebra, with the canonically defined fundamental class $[L] \in QH_*(L;R)$ serving as the unit.

We will also need to use Lagrangian quantum homology with twisted coefficients, see, for instance, \cite{Biran_Cornea_Lag_Top_Enumerative_Geom} and references therein, as well as \cite{Zapolsky_Orientations}. Namely, if $\cQ$ is a local system on $L$ with values in $R$-modules, see \cite{Zapolsky_Orientations}, for example, one can define the twisted homology $QH_*(L;\cQ)$ (\emph{ibid.}). It carries an associative product. It is also unital, provided the $R$-modules comprising $\cQ$ are free and of finite rank, in which case the unit is again given by the well-defined fundamental class $[L]$. Henceforth we assume all local systems to be free and of finite rank.

\begin{rem}
  It will usually be clear from the context over which ring a given local system is defined. When we wish to specify the ring $R$, we will say that we have an $R$-local system.
\end{rem}

The last algebraic structure we need is the so-called quantum action \cite{biran2007quantum}, \cite{Biran_Cornea_Lag_Top_Enumerative_Geom}, \cite{Zapolsky_Orientations}:
$$\bullet\fc QH_*(M;R)\otimes QH_*(L;\cQ) \to QH_*(L;\cQ)\,,$$
which makes $QH_*(L;\cQ)$ into a superalgebra over $QH_*(M;R)$, where $[M]$ acts as the identity.
\begin{rem}
  The quantum action can be shown to equal the composition of the closed-open map (also known as the Albers map) on the first factor, followed by the Lagrangian quantum product.
\end{rem}

\begin{rem}
  A local system on a path-connected space $X$ is completely determined, up to isomorphism, by its fiber $F$ at a chosen point $x_0 \in X$ and the associated monodromy representation $\pi_1(X,x_0) \to\Aut(F)$. This simple fact will be used in the following example.
\end{rem}

\begin{exam}\label{ex:loc_sys_quantum_action}
  These will be important for the proof of Proposition \ref{prop:qm}.
  \begin{enumerate}
    \item A local system $\cQ$ on $L$ is called trivial if the module at each point is $R$, while the parallel transport maps are all identity; in this case we have $QH_*(L;\cQ)\equiv QH_*(L;R)$. We will also refer to this as the untwisted case.
    \item A pivotal role in Lagrangian Floer and quantum homology is played by the so-called superpotential, see \cite{Biran_Cornea_Lag_Top_Enumerative_Geom} and the references therein. We will only need it for Lagrangian tori in $\C P^n$. If $K$ is such a torus, its superpotential is a certain meromorphic function $W$ whose critical points are in a natural bijection with those local systems $\cQ$ on $K$ for which $QH_*(K;\cQ)\neq 0$.
    \item This is an example of the general situation in the previous item. There is a certain torus $T_{\mathrm{Ch}}\subset \C P^n$, which we will refer to as the \emph{Chekanov lift}, which in case $n=2$ is the Chekanov torus described, for instance, in \cite{Chekanov_Schlenk_Notes_Monotone_Twist_Tori}, \cite{Auroux_Mirror_Symm_T_Duality_Compl_Antican_Divisor}, while for $n\geq 3$ it is given by a construction from \cite{Chanda_Hirschi_Wang_Infinitely_Exotic_Lag_Tori_Higher_Projective_Spaces}, also described in detail in \cite{Kawamoto_Shelukhin_Spectral_Invts_Over_Integers}. Its superpotential is explicitly computed in \cite[Example 4.2]{Chanda_Hirschi_Wang_Infinitely_Exotic_Lag_Tori_Higher_Projective_Spaces}. It clearly has critical points, and thus there are local systems $\cQ$ over $\C$ for which $QH_*(T_{\mathrm{Ch}};\cQ)\neq 0$. When $n=2$, one also has $QH_*(T_{\mathrm{Ch}};\F_7)\neq 0$, and moreover for any field $\F$ with $\chr \F\neq 2$ there exist local systems $\cQ$ over $\F$ with $QH_*(T_{\mathrm{Ch}};\cQ)\neq 0$, see \cite[Section 2.6]{leclercq2018spectral}
    \item Let $M=S^2$ and $L\subset M$ be an equator. Let $\cQ_\pm$ be a local system on $L$ with fiber $R$ at a basepoint and monodromy around $L$ being $\pm1\in R^\times=\Aut(R)$. The quantum homology $QH_*(L;\cQ_\pm)$ does not vanish, the basic reason being that $\pm1$ are exactly the critical points of the superpotential associated to $L$. See \cite{Biran_Cornea_Lag_Top_Enumerative_Geom} and references therein. The homology $QH_*(L;\cQ_\pm)$ is a graded module over the ring $R[t,t^{-1}]$, where $|t|=-N_L=-2$, and the quantum action by the point class $\pt \in QH_0(S^2;R)$ on $[L] \in QH_*(L;\cQ_\pm)$ is given by
        $$\pt\bullet[L]=\pm t[L]\,.$$
        This can be extracted from \cite{Biran_Cornea_Lag_Top_Enumerative_Geom}, but the intuition is very simple: there is only one holomorphic disk in $S^2$ with boundary on $L$ of Maslov index $2$ which also passes through a prescribed point in $M\setminus L$, and the orientations work out so that $\pt\bullet[L]=[L]$ in the untwisted case of $\cQ_+$. For $\cQ_-$ the coefficient is multiplied by the monodromy around $L$, which is $-1$.
    \item We can now take the product of two copies of the previous example to obtain the Clifford torus $T=L\times L\subset S^2\times S^2$. The corresponding quantum action is given by the tensor product of the actions in the factors. More specifically, for $\epsilon,\eta\in\{\pm\}$ we can consider the local system
        $$\cQ_{\epsilon\eta}:=\cQ_\epsilon\boxtimes\cQ_\eta\,,$$
        on $T$, where $\boxtimes$ stands for the exterior tensor product.\footnote{The exterior product $\cQ\boxtimes\cR$ of local systems $\cQ,\cR$ on $X,Y$ is a local system on $X\times Y$ which has fibers $(\cQ\boxtimes\cR)_{(x,y)} = \cQ_x\otimes \cR_y$ and obvious parallel transport maps.} Explicitly, $\cQ_{\epsilon\eta}$ is a local system with fiber $R\otimes_RR\equiv R$ at a basepoint, having monodromy $\epsilon\cdot 1$ around loops of the form $L\times \pt$, and monodromy $\eta\cdot 1$ around loops of the form $\pt\times L$. Then
        $$QH_*(T;\cQ_{\epsilon\eta})\cong QH_*(L;\cQ_\epsilon)\otimes_{R[t,t^{-1}]} QH_*(L;\cQ_\eta)$$
        as algebras. Moreover, for the quantum action of $\pt \in QH_0(S^2\times S^2;R)$ on $[T] = [L]\otimes[L] \in QH_*(T;\cQ_{\epsilon\eta})$ we have
        \begin{multline*}
          \pt\bullet[T]=(\pt_{S^2}\otimes\pt_{S^2})\bullet([L]\otimes[L]) = (\pt_{S^2}\bullet[L])\otimes(\pt_{S^2}\bullet[L])\\
          =(\epsilon t[L])\otimes(\eta t[L]) = \epsilon\eta t^2[L]\otimes[L]=\epsilon\eta t^2[T]=\epsilon\eta q[T]\,.
        \end{multline*}
        The last equality is thanks to the fact that one can view $R[q,q^{-1}]$ as a subring of $R[t,t^{-1}]$ by means of the identity $t^2=q$, see \cite{biran2007quantum}. In particular, we have the following consequence in case $2\in R$ is invertible:
        \begin{multline*}
          e^R_\pm\bullet [T]=\frac{[M]\pm q^{-1}\pt}{2}\bullet [T]=\frac 12([M]\bullet[T] \pm  q^{-1}\pt\bullet [T])\\
           = \frac 12(1\pm \epsilon\eta q^{-1}q)[T] = \frac 12(1\pm \epsilon\eta)[T]\,.
        \end{multline*}
        This equals $[T]$ if $\pm\epsilon\eta=1$ and zero if $\pm\epsilon\eta=-1$.
  \end{enumerate}

\end{exam}

In \cite{leclercq2018spectral}, the Lagrangian spectral invariant corresponding to $L$, $R$, and $\cQ$ is constructed as a function
$$QH_*(L;\cQ)\times C^\infty(M\times[0,1])\to\R\cup\{-\infty\}\,.$$
In this paper we will only use the invariant corresponding to the fundamental class $[L] \in QH_*(L;\cQ)$. Since $[L]$ is the unit element, we have $[L]\neq 0$ if and only if $QH_*(L;\cQ)\neq 0$. In this case the corresponding spectral invariant is always finite, and we will denote the induced function by
$$\ell^\cQ_L \fc C^\infty(M\times[0,1])\to \R\,.$$
In case $\cQ$ is the trivial local system, we will use the notation $\ell^R_L$.

Here we need certain properties satisfied by these functions, which we will now formulate. For this we need some preliminary definitions. Consider the path space
$$\cP_L = \{\gamma \in C^\infty([0,1],M)\,|\,\gamma(0),\gamma(1) \in L\,,[\gamma]=*\in\pi_1(M,L)\}\,,$$
where $* \in \pi_1(M,L)$ denotes the class of the constant path. Recall that a capping of $\gamma \in \cP_L$ is a smooth map $u \fc D^2_+ \to M$, where $D^2_+ = D^2\cap\{\IM z \geq 0\}$, with $u([-1,1])\subset L$ and $u(e^{\pi i t}) = \gamma(t)$ for all $t \in [0,1]$. For a Hamiltonian $H$ on $M$ we have the corresponding action functional
$$\cA_{H:L}(\gamma,u) = \int_0^1H_t(\gamma(t))\,dt - \int_{D^2_+}u^*\omega\,.$$
Its critical points are pairs $(\gamma,u)$ where $\gamma$ is an orbit of $H$, that is $\dot\gamma(t) = X^t_H(\gamma(t))$. We let $\cP_L(H) \subset \cP_L$ be the collection of such orbits. Then the action spectrum of $H$ relative to $L$ is defined by
$$\Spec(H:L) = \{\cA_{H:L}(\gamma,u)\,|\,\gamma \in \cP_L(H)\,,u\text{ a capping of }\gamma\}\,.$$
If $\gamma$ is a nondegenerate orbit, meaning $d_{\gamma(0)}\phi_H(T_{\gamma(0)})L$ intersects $T_{\gamma(1)}L$ transversely, then for any capping $u$ of $\gamma$ we have the corresponding Viterbo--Malsov index \cite{Viterbo_Intersection_Sous_Varietes_Lagrangiennes}
$$\mathsf m_{H:L}(\gamma,u) \in \Z\,.$$
We use the convention where for $H$ which in a Weinstein neighborhood $U$ of $L$ equals $\pi^*f$, where $\pi \fc U \to L$ is the cotangent bundle projection, and where $f \in C^\infty(L)$ is a $C^2$-small Morse function, the index $\mathsf m_{H:L}(\gamma,u)$ equals the Morse index of $f$ at $\gamma$, considered as a critical point of $f$, in case $u$ is the constant capping. When recapping by a disk $w$, we have the relation
$$\mathsf m_{H:L}(\gamma,u\sharp w) = \mathsf m_{H:L}(\gamma,u) - \mathsf m([w])\,.$$

We can now list the relevant properties of the Lagrangian spectral invariants, see \cite{leclercq2018spectral}.
\begin{itemize}
  \item \tb{(invariance):} $\ell^\cQ_L(H)$ only depends on $\wt\phi_H \in \wt\Ham(M,\omega)$, provided $H$ is normalized, that is $\int_MH_t\omega^n=0$ for all $t$; by abuse of notation, we let $\ell^\cQ_L$ also stand for the induced function
      $$\ell^\cQ_L\fc\wt\Ham(M,\omega) \to \R\,.$$
      It will be clear from the context which one of the two functions is meant in each instance.
  \item \tb{(spectrality):} $\ell^\cQ_L(H) \in \Spec(H:L)$; moreover if $\phi_H(L)$ intersects $L$ transversely, then there are $\gamma \in \cP_L(H)$ and a capping $u$ of $\gamma$ such that $\mathsf m_{H:L}(\gamma,u)= \dim L$ and such that $\ell^\cQ_L(H) = \cA_{H:L}(\gamma,u)$.
  \item \tb{(quantum action):} if $A \in QH_*(M;R)$ is such that $A\bullet [L] = [L]$, then $\ell^\cQ_L \leq c^R(A;\cdot)$ as functions on $\wt\Ham(M,\omega)$.
  \item \tb{(Lagrangian control):} if $H|_L\equiv c \in \R$, then $\ell^\cQ_L(H) = c$.
  \item \tb{(normalization):} if $c$ is a function of time, then $\ell^\cQ_L(H+c) = \ell^\cQ_L(H) + \int_0^1c(t)\,dt$.
  \item \tb{(positivity):} $\ell^\cQ_L(H) + \ell^\cQ_L(\ol H) \geq 0$, where $\ol H_t=-H_{1-t}$.
  \item \tb{(continuity):} $\int_0^1 \min (H_t-H'_t)\,dt \leq \ell^\cQ_L(H) - \ell^\cQ_L(H') \leq \int_0^1 \max(H_t-H'_t)\,dt$.
\end{itemize}

\subsection{Proof of Proposition \ref{prop:qm}}

Recall what the proposition says: \emph{If $(M,\omega)$ is as in Theorem \ref{thm:main} and $T\subset M$ is the Clifford torus, then there are a Hamiltonian $H\in C^\infty(M,[0,1])$, and a quasi-morphism $\nu \fc \Ham(M,\omega) \to \R$, which is $2$-Lipschitz with respect to the Hofer distance, such that $\nu(\phi_H^t)=t$ for all $t\in \R$, and $\nu(\psi)=0$ whenever $\psi(T)=T$.}

The fundamental idea in the construction of such a quasi-morphism is contained in \cite{Khanevsky_Hofers_Metric_Space_Diameters}, and goes back at least to
\cite{Entov_Polterovich_Py_Continuity_QMs_Sympl_Maps}, and it is \emph{to take the difference of two suitable quasi-morphisms}. More specifically, we will see that there are Lagrangians $K,L\subset M$ such that $K\cap L = \varnothing$, and quasi-morphisms $\mu_K,\mu_L \fc \Ham(M,\omega)\to\R$ coming from Hamiltonian spectral invariants---these are $\mu^\F,\mu^\F_\pm$ for suitable coefficient fields $\F$---such that $\mu_K$ ``detects'' $K$ in the sense that $\mu_K(\phi_G^t)=t$ for any $G\in C^\infty(M)$ with $G|_K\equiv 1$, and similarly for $\mu_L$, which ``detects'' $L$, and such that $\mu_K$, $\mu_L$ coincide on the subgroup $\{\psi\in\Ham(M,\omega)\,|\,\psi(T)=T\}$. We then put $\nu=\mu_K-\mu_L$, pick any $H\in C^\infty(M,[0,1])$ with $H|_K\equiv 1$, $H|_L\equiv 0$, and the required properties easily follow.

Let us now pass to the details. To flesh out the above idea, we need a few preparatory results. The first one of these relates Hamiltonian and Lagrangian spectral invariants in the context of quasi-morphisms.
\begin{prop}\label{prop:Ham_sp_invt_QMs_Lag_sp_invt}
  Let $(M,\omega)$ be a closed connected symplectic manifold and $L\subset M$ be Lagrangian. Assume $\F$ is a field and let $\cQ$ be a local system on $L$ such that $QH_*(L;\cQ)\neq 0$. If $A \in QH_*(M;\F)$ is a class satisfying $A\bullet [L] = [L]$, such that the corresponding spectral invariant $c^\F(A,\cdot) \fc \wt\Ham(M,\omega)\to\R$ is a quasi-morphism, then
  $$c^\F(A,\cdot) - \ell^\cQ_L$$
  is a bounded function on $\wt\Ham(M,\omega)$. In particular $\ell^\cQ_L$ is likewise a quasi-morphism, and moreover the homogenizations of $c^\F(A,\cdot)$, $\ell^\cQ_L$ coincide.
\end{prop}
\begin{proof}
  Let us abbreviate $c:=c^\F(A;\cdot)$. Fix $\wt\phi \in \wt\Ham(M,\omega)$. Note that by the quasi-morphism property of $c$ we have:
  $$c(\id) - c(\wt\phi) - c(\wt\phi^{-1}) \geq -D(c)\,,\quad\text{which implies}\quad -c(\wt\phi^{-1})\geq c(\wt\phi) - c(\id) - D(c)\,.$$
  Now we have the following string of inequalities:
  \begin{align*}
    c(\wt\phi) &\geq \ell^\cQ_L(\wt\phi) && \text{by quantum action}\\
               &\geq -\ell^\cQ_L(\wt\phi^{-1}) && \text{by positivity}\\
               &\geq -c(\wt\phi^{-1}) && \text{by quantum action}\\
               &\geq c(\wt\phi) - c(\id) - D(c) && \text{by the above discussion}\,.
  \end{align*}
  In total we deduce:
  $$c(\wt\phi) - c(\id) - D(c) \leq \ell^\cQ_L(\wt\phi) \leq c(\wt\phi)\,,$$
  as claimed. The last assertions then obviously follow.
\end{proof}

In the next result we will apply this general statement to the specific symplectic manifolds appearing in Theorem \ref{thm:main}, and certain Lagrangians, which we will now list.
\begin{itemize}
  \item In $M=\C P^n$, we will need $\R P^n$, the Clifford torus $T$, as well as the Chekanov lift $T_{\mathrm{Ch}}$, see Example \ref{ex:loc_sys_quantum_action}, item (iii).
  \item In $M=S^2\times S^2$ the required Lagrangians are:
        $$\text{the antidiagonal}\quad\ol \Delta = \{(z,-z)\,|\,z\in S^2\}\,,$$
        $$\text{the exotic torus}\quad T_{\mathrm{ex}} = \{(z,w) \in M\,|\,\langle z,w\rangle = -\tfrac 12\,,z_3+w_3=0\}\,,$$
      where $z_3,w_3$ denote the third coordinates of $z,w\in S^2\subset \R^3$, and the Clifford torus $T$.
\end{itemize}

\begin{coroll}\label{cor:Ham_qm_coincides_homog_Lag_sp_invt}
  \begin{enumerate}
    \item If $M=\C P^n$, $L\subset M$ is a Lagrangian, $\F$ is a field and $\cQ$ is an $\F$-local system on $L$ with $QH_*(L;\cQ) \neq 0$, then $c^\F_+ - \ell^\cQ_L$ is a bounded function on  $\wt\Ham(\C P^n)$; in particular $\ell^\cQ_L$ is a quasi-morphism whose homogenization $(\ell^\cQ_L)_\infty$ coincides with $\mu^\F$, and thus descends to $\Ham(\C P^n)$. This applies to the following cases:
        \begin{enumerate}
          \item $L=\R P^n$, $\F = \F_2$, $\cQ=$ the trivial local system;
          \item $L= T_{\mathrm{Ch}}$ local systems described in Example \ref{ex:loc_sys_quantum_action}, item (iii).
          \item $L=T$, the Clifford torus, $\F$ is any field and $\cQ$ is the trivial local system.
        \end{enumerate}
    \item Let $M=S^2 \times S^2$, $\F$ be a field with $\chr \F \neq 2$. Assume $L$ is a Lagrangian and $\cQ$ is a local system with $QH_*(L;\cQ)\neq 0$ such that $e^\F_+\bullet [L] = [L]$. Then $c^\F(e^\F_+,\cdot) - \ell^\cQ_L$ is bounded. In particular $\ell^\cQ_L$ is a quasi-morphism satisfying $(\ell^\cQ_L)_\infty = \mu^\F_+$, and thus $(\ell^\cQ_L)_\infty$ descends to $\Ham(S^2\times S^2)$. The same result holds if we replace $e^\F_+$, $c^\F(e^\F_+,\cdot)$, $\mu^\F_+$ by $e^\F_-$, $c^\F(e^\F_-,\cdot)$, $\mu^\F_-$, respectively. This applies to the following cases for $\F=\F_3$:
        \begin{enumerate}
          \item $L = \ol\Delta$, $\cQ=$ the trivial local system, and $e^\F_-$;
          \item $L = T_{\mathrm{ex}}$, $\cQ=$ the trivial local system, and $e^\F_+$;
          \item $L = T$, the Clifford torus, the local systems are $\cQ_{++}$ and $\cQ_{+-}$ from Example \ref{ex:loc_sys_quantum_action}, for which it holds that $e^{\F_3}_\pm\bullet [T]=[T]$ for $[T]\in QH_*(T;\cQ_{+\pm})$.
        \end{enumerate}
  \end{enumerate}
\end{coroll}
\begin{proof}
  The general statement in (i) is a direct consequence of Proposition \ref{prop:Ham_sp_invt_QMs_Lag_sp_invt}, since $c^\F_+$ is the spectral invariant relative to $[M]$ and we always have $[M]\bullet [L] = [L]$. The descent to $\Ham(\C P^n)$ is a consequence of Proposition \ref{prop:EP_mu_Lipschitz_qms_descend_Ham}. Let us now consider the specific cases. For (a) note that $QH_*(\R P^n;\F_2) \neq 0$ \cite{biran2007quantum}. For (b) see Example \ref{ex:loc_sys_quantum_action}, item (iii). For (c) we have $QH_*(T;\F) \neq 0$ for any field $\F$, see \cite[Section 8.1]{Biran_Cornea_Lag_Top_Enumerative_Geom}, where the corresponding superpotential is computed.

  The general statement in (ii) is again a consequence of Propositions \ref{prop:Ham_sp_invt_QMs_Lag_sp_invt}, \ref{prop:EP_mu_Lipschitz_qms_descend_Ham}. Let us now consider the specific cases. For (a), it is known that $\ol\Delta$ is $e^\F_-$-superheavy\footnote{In the terminology of Entov--Polterovich \cite{Entov_Polterovich_Rigid_subsets}.} \cite{Eliashberg_Polterovich_Sympl_QS_Quadric_Surface}, which, as was shown in \cite{leclercq2018spectral}, implies that $e^\F_-\bullet [\ol\Delta] = [\ol \Delta]$. For (b), \emph{ibid.}, it is shown that $e^\F_+\bullet [T_{\mathrm{ex}}] = [T_{\mathrm{ex}}]$. For (c) we need only reference the calculation in Example \ref{ex:loc_sys_quantum_action}, item (v).
\end{proof}

We will also need the following lemma, which is a boundedness result for differences of Lagrangian spectral invariants.
\begin{lemma}\label{lem:boundedness_Lag_sp_invt}Let $(M,\omega)$ be a closed connected symplectic manifold and let $L$ be a Lagrangian with monotonicity constant $\tau$. Let $R,R'$ be arbitrary ground rings, let $\cQ,\cQ'$ be local systems over those rings, and suppose $QH_*(L;\cQ)\neq 0\neq QH_*(L;\cQ')$. Then  for any Hamiltonian $G^0$ on $M$ with $\phi_{G^0}(L)=L$ we have $|\ell_L^{\cQ}(G^0)-\ell_L^{\cQ'}(G^0)|\leq \tau\cdot\dim L$ .
\end{lemma}
\begin{proof}
  Let $n=\dim L = \frac 12 \dim M$. Denote $\ell=\ell^\cQ_L$, $\ell'=\ell^{\cQ'}_L$. Let $G^0$ be a Hamiltonian such that $\phi_{G^0}(L) = L$, which we can assume satisfies $G^0_t\equiv 0$ for $t$ near $1$, by suitable time reparametrization. Let $a = \ell(G^0)$.

  Fix $\epsilon > 0$ and let $G$ be a perturbation of $G^0$ obtained as follows. Suppose $\delta > 0$ is such that $G^0_t\equiv 0$ for $t\in[1-\delta,1]$. Let $f \in C^\infty(L)$ be a Morse function, and let $H \in C^\infty(M)$ be obtained by suitably cutting off $\pi^*f$, where $\pi$ is the bundle projection on a Weinstein neighborhood of $L$, so that locally in that neighborhood, $\phi^t_H(L) = \phi_{\pi^*f}^t(L)$ for $t \in [0,1]$. Now time-reparametrize $H$ to obtain a Hamiltonian $H'$ satisfying $H'_t\equiv 0$ for $t \leq 1-\delta$ and replace the zero portion of $G^0_t$ for $t \in [1-\delta,1]$ by $H'$. Let us denote the resulting Hamiltonian by $G$. Suitably rescaling $f$, we can achieve $\|H'\|_{C^0} < \epsilon$, and thus $\|G^0-G\|_{C^0} < \epsilon$.

  We have $|\ell(G) - a| = |\ell(G) - \ell(G^0)| < \epsilon$ by continuity. By spectrality, $a \in \Spec(G^0:L)$, whence
  $$\Spec(G^0:L) = a + \mathsf A\cdot \Z\,,$$
  where $\mathsf A >0$ is the positive generator of the subgroup $\langle \omega,\pi_2(M,L)\rangle \subset \R$. Thanks to our construction of the perturbation $G$, the set $\cP_L(G)$ consists exactly of those curves $\gamma \fc [0,1] \to M$ which satisfy $\gamma(t) = \phi_{G^0}^t(\gamma(0))$ and $\gamma(1) \in \Crit f$. A straightforward action calculation then implies that
  $$\Spec(G:L) \subset (a-\epsilon,a+\epsilon) + \mathsf A \cdot \Z\,.$$
  Let us refer to the capped orbits of $G$ whose actions are contained in an interval $(b,c)$ as orbits of $G$ in the window $(b,c)$. By construction, since $G$ is nondegenerate in the sense that $\phi_G(L)$ intersects $L$ transversely, there exists $i \in \Z$ such that the Viterbo--Maslov indices of the orbits of $G$ in the window $(a-\epsilon,a+\epsilon)$ all lie in the set
  $$\{i,i+1,\dots,i+n\}\,.$$
  By spectrality, $\ell(G)$, being in the interval $(a-\epsilon,a+\epsilon)$, is the action relative to $G$ of one of the orbits in this window, which moreover has index $n$, whence
  $$n \in \{i,\dots,i+n\}\,,$$
  which implies $i \in \{0,\dots,n\}$. Since $L$ is monotone, the orbits of $G$ in the window $(a-\epsilon,a+\epsilon)+k\mathsf A$, where $k \in \Z$, all have indices in the set
  $$\{i,\dots,i+n\} + kN_L\,.$$
  For $|k|>\frac{n}{N_L}$ we have
  $$n \notin \{i,\dots,i+n\} + kN_L\,.$$
  Again by spectrality, $\ell'(G)$ is the action of an orbit of $G$ of index $n$, and from the above it follows that it must lie in a window $(a-\epsilon,a+\epsilon)+k\mathsf A$ where $|k|\leq \frac{n}{N_L}$. It follows that
  $$|\ell'(G) - a| < \epsilon + \frac{n}{N_L}\cdot\mathsf A\,.$$
  Taking $\epsilon \to 0$, we have $\ell'(G) \to \ell'(G^0)$ and therefore
  $$|\ell(G^0) - \ell'(G^0)| \leq \frac{n\mathsf A}{N_L}=\tau\cdot\dim L\,,$$
  where we used the fact that $\tau=\frac{\mathsf A}{N_L}$. The lemma is proved.
\end{proof}

In the following lemma we describe precise conditions under which there is a quasi-morphism on $\Ham(M,\omega)$ with properties asserted in Proposition \ref{prop:qm}.
\begin{lemma}
  Let $(M,\omega)$ be a closed connected symplectic manifold and let $K,L,T\subset M$ be Lagrangians with $K\cap L = \varnothing$. Assume there are fields $\F_K,\F_L$, local systems $\cQ_K, \cQ_K'$ over $\F_K$ and $\cQ_L,\cQ_L'$ over $\F_L$, where $\cQ_K$ is defined on $K$, $\cQ_L$ is defined on $L$ while $\cQ_K',\cQ_L'$ are defined on $T$. Let $\mu_K,\mu_L\fc\Ham(M,\omega) \to \R$ be homogeneous quasi-morphisms which are $1$-Lipschitz with respect to the Hofer distance, such that:
  $$\mu_K=(\ell^{\cQ_K}_K)_\infty=\big(\ell_T^{\cQ_K'}\big)_\infty\qquad \text{and} \qquad \mu_L=(\ell^{\cQ_L}_L)_\infty=\big(\ell_T^{\cQ_L'}\big)_\infty\,.$$
  Then $\nu:=\mu_K-\mu_L$ is a homogeneous quasi-morphism on $\Ham(M,\omega)$ which is $2$-Lipschitz with respect to the Hofer distance and for any $H\in C^\infty(M)$ such that $H|_K\equiv 1$ and $H|_L\equiv 0$ we have $\nu(\phi^t_H)=t$, and moreover $\nu(\psi) = 0$ for all $\psi \in \Ham(M,\omega)$ with $\psi(T) = T$.
\end{lemma}
\begin{proof}
  It is obvious from the definition of $\nu$ that it is a homogeneous quasi-morphism which is $2$-Lipschitz with respect to the Hofer distance. Let $H$ be as in the statement of the lemma.

  Let us abbreviate $\ell_K:=\ell^{\cQ_K}_K$, $\ell_L:=\ell^{\cQ_L}_L$. From the assumptions it follows that $\ell_K$ is a real-valued function, thus in particular $QH_*(K;\cQ_K)\neq 0$, which implies, by Lagrangian control, that for all $t \in \R$ we have
  $$\ell_K(tH) = t\,.$$
  Putting
  $$\langle H \rangle :=\frac{\int_M H\omega^n}{\int_M\omega^n}\,,$$
  where $n=\frac12\dim M$, we see that $tH-t\langle H\rangle$ is normalized, and thus by the invariance and normalization properties we have
  $$\ell_K(\wt\phi_{tH})=\ell_K(tH-t\langle H\rangle) = \ell_K(tH) - t\langle H\rangle = t-t\langle H \rangle\,.$$
  Similarly, since $\ell_L$ is real-valued, $QH_*(L;\cQ_L)\neq 0$, which implies that $\ell_L(tH) = 0$, again by Lagrangian control. Analogously to the last argument, it follows that
  $$\ell_L(\wt\phi_{tH}) = \ell_L(tH) - t\langle H\rangle = -t\langle H\rangle\,.$$
  We thus have
  $$\mu_K(\phi_{tH})=(\ell_K)_\infty(\wt\phi_{tH})=\lim_{k\to\infty}\frac{\ell_K(\wt\phi_{tH}^k)}{k}=\lim_{k\to\infty}\frac{\ell_K(\wt\phi_{ktH})}{k} = t-t\langle H\rangle\,,$$
  and similarly
  $$\mu_L(\phi_{tH})=(\ell_L)_\infty(\wt\phi_{tH})=-t\langle H\rangle\,.$$
  Thus
  $$\nu(\phi^t_H)=\nu(\phi_{tH}) = \mu_K(\phi_{tH}) - \mu_L(\phi_{tH}) = t\,,$$
  as claimed.
  For the last assertion, let $G$ be a Hamiltonian such that $\phi_G(T)=T$. It follows from Lemma \ref{lem:boundedness_Lag_sp_invt} that $|\ell^{\cQ_K'}_T(G) - \ell^{\cQ_L'}_T(G)| \leq n\tau$, where $\tau$ is the monotonicity constant of $T$. Thus the restriction of $\nu=\mu_K-\mu_L = (\ell^{\cQ_K'}_T)_\infty-(\ell^{\cQ_L'}_T)_\infty$ to $\{\phi\in\Ham(M,\omega)\,|\,\phi(T)=T\}$ is bounded in absolute value by $n\tau$, and, since it is a homogeneous quasi-morphism, it therefore vanishes, which proves the last assertion of the lemma.
\end{proof}

With all this preparation, we can now prove Proposition \ref{prop:qm}.
\begin{proof}[Proof of Proposition \ref{prop:qm}]
  The existence of $\nu$ is a consequence of the preceding lemma, combined with the following observations.
  \begin{enumerate}
    \item For $M=\C P^n$, take $K=\R P^n$, $L=T_{\mathrm{Ch}}$, $T=$ the Clifford torus as the Lagrangians, and note that $\R P^n$ and $T_{\mathrm{Ch}}$ are indeed disjoint \cite{Kawamoto_Shelukhin_Spectral_Invts_Over_Integers}. Now let $\F_K=\F_2$, $\F_L=\C$, let $\cQ_K = \cQ_K'$ be the trivial local system, $\cQ_L$ be one of the local systems for which $QH_*(T_{\mathrm{Ch}};\cQ_L)\neq 0$ (Example \ref{ex:loc_sys_quantum_action}, item (iii)), and $\cQ_L'$ be the trivial local system. The quasi-morphisms are $\mu_K=\mu^{\F_2}$ and $\mu_L=\mu^{\C}$. Then the assumptions of the lemma hold as a direct consequence of Corollary \ref{cor:Ham_qm_coincides_homog_Lag_sp_invt}.
    \item In case $M=S^2\times S^2$, let $K=T_{\mathrm{ex}}$, $L = \ol\Delta$, $T=$ the Clifford torus be the Lagrangians, where we note that $\ol\Delta$ and $T_{\mathrm{ex}}$ are obviously disjoint. The fields are $\F_K=\F_L=\F_3$, while the local systems are: $\cQ_K=\cQ_L=$ trivial, $\cQ_K'=\cQ_{++}$, $\cQ_L' = \cQ_{+-}$, see Example \ref{ex:loc_sys_quantum_action}, item (v). The quasi-morphisms are $\mu_K=\mu_+^{\F_3}$ and $\mu_L=\mu_-^{\F_3}$. Again, the assumptions of the lemma hold as a consequence of Corollary \ref{cor:Ham_qm_coincides_homog_Lag_sp_invt}.
  \end{enumerate}
  For the Hamiltonian we can take any $H \in C^\infty(M,[0,1])$ with $H|_K\equiv 1$ and $H|_L\equiv 0$.
\end{proof}

\bibliographystyle{alpha}
\bibliography{bibfile}

\noindent
\begin{tabular}{l}
University of Haifa \\
Department of Mathematics \\
Faculty of Natural Sciences \\
3498838, Haifa, Israel \\
\&\\
MI SANU \\
Kneza Mihaila 36 \\
Belgrade 11001\\
Serbia\\
{\em E-mail:}  \texttt{frol.zapolsky@gmail.com}
\end{tabular}

\end{document}